\documentclass[11pt]{article}

\usepackage[utf8]{inputenc}
\usepackage{color}
\usepackage{amsmath}
\usepackage{amsthm}
\usepackage{mathrsfs}
\usepackage{stmaryrd}
\usepackage{microtype}
\usepackage[T1]{fontenc}
\usepackage{authblk}
\newcommand{\R}{\mathbb{R}}

\DeclareMathOperator*{\dvg}{\mbox{div}}

\usepackage{amssymb,amsfonts}
\usepackage{bbold}
\usepackage{multirow}

\usepackage{blindtext}
\usepackage{hyperref}
\hypersetup{
    colorlinks=true,
    linkcolor=green,
}
\newtheorem{theorem}{Theorem}[section]
\newtheorem{proposition}{Proposition}[section]

\newtheorem{lemma}{Lemma}

\theoremstyle{remark}
\newtheorem{remark}{Remark}
\theoremstyle{definition}

\begin{document}
\title{A Local Energy Identity for Parabolic Equations with Divergence-Free Drift}

\author{Francis Hounkpe\footnote{Email adress: \texttt{hounkpe@maths.ox.ac.uk};}}
\affil{Mathematical Institute, University of Oxford, Oxford, UK}

\renewcommand\Affilfont{\itshape\large}

\maketitle
\begin{abstract}
We prove a local energy identity for a class of distributional solutions, in $L_{2,\infty} \cap W^{1,0}_2$, of parabolic equations with divergence-free drift. 
\end{abstract}

\section{Introduction}

We are considering the parabolic equations of the type
\[
\partial_t u - \dvg(a\nabla u) + b.\nabla u = 0, 
\]
where $a$ is a bounded, symmetric and uniformly elliptic matrix and $b$ a divergence-free vector field belonging to $L_{\infty}(BMO^{-1})$. We say that a divergence-free vector field $b$ belongs to the space $BMO^{-1}$ if there exits a skew symmetric matrix $d$ belonging to $BMO$ such that $b = \dvg(d)$. Therefore, the above equation can be rewritten as follows: 
\begin{equation}\label{E1}
\partial_t u - \dvg(A\nabla u) = 0,
\end{equation}
where $A= a + d$, with $a$ as before and $d \in L_{\infty}(BMO)$ a skew symmetric matrix.

G. Seregin and co-authors introduced, in their paper \cite{Ser12}, the notion of \textit{suitable weak solutions} to equation~\eqref{E1}, which are distributional solutions that belong to the energy class $L_{2,\infty} \cap W^{1,0}_2$ and that satisfy a particular local energy inequality. In this paper, we establish a local energy identity for distributional solutions of \eqref{E1} which belong to the energy class $L_{2,\infty} \cap W^{1,0}_2$, and therefore, we prove at the same time that the local energy inequality required in the definition of \textit{suitable weak solutions}, introduced in \cite{Ser12}, is a direct consequence of being a distributional solution in the above energy class.

\section{Preliminaries}
In what follows, we will use the following abbreviated notations:
$B := B(0,1)$ (the unit ball of $\R ^n$), $Q := B \times (-1,0)$, as well as $z:= (x,t)$.

We recall that a function $d$ is in the space $BMO(\Omega;\R ^{n\times n})$ if the following quantity 
\[
\sup\left\{ \frac{1}{|B(0,r)|}\int_{B(x_0,r)} |d-[d]_{x_0,r}| dx : B(x_0,r) \subset\subset \Omega \right\},
\]
with $[d]_{x_0,r}$ the average of $d$ over $B(x_0,r)$, is bounded; and a function $u$ belongs to the Hardy space $\mathcal{H}^1(\R^n)$ if there exists a function $\phi \in C^{\infty}_0(B)$ such that
\[
u_{\phi} \in L_1(\R^n),
\]
where $u_{\phi}(x) := \sup_{t>0} |(\phi_t\star u)(x)|$,
and $\phi_t(x) := t^{-n}\phi(x/t)$.\\
For simplicity we adopt the following notation convention $\partial_i f = {f,}_i$. We have the following classical div-curl type lemma for Hardy spaces, which is a direct consequence of Theorem II.1 in \cite{Coif93}.
\begin{lemma}\label{L1}
Let $u \in W^1_p(\R^n)$ and $v\in W^1_q(\R^n)$, with $1<p<\infty$ and $1/p + 1/q=1$. Then 
$u,_{j} v,_{i} - v,_{j} u,_{i} \in \mathcal{H}^1(\R^n)$ for all $i,j = 1,\ldots,n$ and we have 
\[
\|u,_{j} v,_{i} - v,_{j} u,_{i}\|_{\mathcal{H}^1} \leq C \|\nabla u\|_{L_p} \|\nabla v\|_{L_q}, \quad \forall i,j=1,\ldots,n.
\]
\end{lemma}

We recall also some basic facts related to the spectral decomposition of the Laplace operator on a bounded domain $\Omega$ of $R^n$, with smooth boundary. The Laplacian viewed as an unbounded operator from $L_2(\Omega)$ into itself has a discrete spectrum; we denote by $0<\lambda_1<\lambda_2<\ldots<\lambda_n<\ldots$ (with $\lambda_n \to \infty$), its eigenvalues and $\{\phi_k\}_{k= 1}^{\infty}$ the corresponding eigenvectors which form a Hilbert basis of $L_2(\Omega)$. Setting $\mathring{L}^1_2(\Omega)$ to be the completion of $C^{\infty}_0(\Omega)$ with respect to the Dirichlet semi-norm $\|u\|_{L^1_2}^2:= \displaystyle\int |\nabla u|^2$, and $H^{-1}(\Omega)$ to be the dual of $\mathring{L}^1_2(\Omega)$, we have the following classical lemma, which gives us a Hilbert basis of $\mathring{L}^1_2(\Omega)$ and a representation of the norm of $H^{-1}$ by means of the eigenvectors and eigenvalues of the Laplace operator.   
\begin{lemma}\label{L2}
$(\phi_k/\sqrt{\lambda_k})_{k=1}^{\infty}$ is a Hilbert basis of $\mathring{L}^1_2(\Omega)$ and as a direct consequence, we have that, if $f\in H^{-1}(\Omega)$, then
\[
\|f\|^2_{H^{-1}(\Omega)} = \sum_{k=1}^{\infty} f_k^2/\lambda_k,
\]
where $f_k = \langle f,\phi_k\rangle$.
\end{lemma}
\begin{proof}
The proof of this result is quiet classical, therefore we skip it.
\end{proof}
\section{Main Theorem}
We now state the main result of this paper.
\begin{theorem}\label{T1}
Let $u$ belonging to the energy class
\[ L_{2,\infty}(Q) \cap W^{1,0}_2(Q), \] such that 
\begin{equation}\label{E3}
\int_Q u\partial_t \phi dz = \int_Q (A\nabla u).\nabla \phi dz \quad \forall \phi \in C^{\infty}_0(Q),
\end{equation}
where $A=a+d$, with $a\in L_{\infty}(Q;\R^{n\times n})$ a symmetric matrix satisfying \[ \nu \mathbb{I}_n \leq  a \leq \nu^{-1} \mathbb{I}_n\] and $d\in L_{\infty}(-1,0;BMO(B;\R^{n\times n}))$ a skew symmetric matrix. Then the following energy identity holds for all $t_0 \in (-1,0)$ and for all test functions $\phi\in C^{\infty}_0(B\times (-1,1))$:
\begin{multline*}
\frac{1}{2}\int_{B} \phi(x,t_0)|u(x,t_0)|^2 dx + \int_{-1}^{t_0}\int_{B}\phi \nabla u.a\nabla u dz = \frac{1}{2} \int_{-1}^{t_0}\int_{B} |u|^2\partial_t \phi dz\\ - \int_{-1}^{t_0}\int_{B} (A\nabla u).\nabla\phi u dz.
\end{multline*}
\end{theorem}

\section{Proof of Theorem~\ref{T1}}
The method we use for this proof are due to Seregin, in his lecture notes:"Parabolic Equations".\\
We start by proving a simple regularity result for the time derivative of $u$ defined as in Theorem~\ref{T1}. 
\begin{proposition}\label{P1}
Let $u$ defined as in Theorem~\ref{T1}. Then 
\[ \partial_t u \in L_2(-1,0;H^{-1}(B)) \]
\end{proposition}
\begin{proof}
\textbf{Step 1.} Let us set 
\[ g(x,t) = A(x,t)\nabla u (x,t), \] and consider the problem
\begin{equation}\label{E2}
\begin{cases}
-\Delta U (\cdot,t) = \dvg g (\cdot,t)\quad \mbox{for a.e $t\in (-1,0)$}\\
U(\cdot,t)|_{\partial B} = 0.
\end{cases}
\end{equation}
Let $v \in C^{\infty}_0(B)$, we have:
\begin{align*}
\int_B \dvg g(\cdot,t)v dx &= - \int_B g(\cdot,t).\nabla v dx\\
&= -\int_B (a\nabla u).\nabla v dx - \int_B (d\nabla u).\nabla v dx\\
&=: A_1 + A_2.
\end{align*}
We have by a straightforward computation that 
\[ |A_1| \leq \|a\|_{L_{\infty}(Q)}\|\nabla u(\cdot,t) \|_{L_2(B)}\|\nabla v \|_{L_2(B)} \quad \mbox{for a.e $t\in (-1,0)$}.\]
On the other hand, we have thanks to the skew symmetry of $d$, that $A_2$ can be rewritten as follows
\[ -A_2 =  \frac{1}{2} \sum_{i,j=1}^n\int_B d_{ij}(u,_j v,_i - v,_j u,_i) dx. \] Denote by $\bar{u}$ the extension of $u$ from $B$ to $\R^n$ such that
\[ \|\bar{u}(\cdot,t)\|_{W^1_2(\R ^n)} \leq c \|u(\cdot,t)\|_{W^1_2(B)} \quad \mbox{for a.e $t\in (-1,0)$}, \] where $c$ depends only on $n$. Similarly, let us denote by $\bar{d}$ the extension of $d$ from $B$ to $\R^n$ such that
\[ \|\bar{d}(\cdot,t)\|_{BMO(\R^n;\R^{n\times n})} \leq c\|d(\cdot,t)\|_{BMO(B;\R^{n\times n})} \quad \mbox{for a.e $t\in (-1,0)$},\] where, again, $c$ depends only on $n$. In the later case, to construct such an extension, one can use a reflection on the boundary (See, e.g., Theorem 2 in \cite{Jon80}, where this is done for very general domains $\Omega \subset \R^n$). Therefore, because $v$ is compactly supported in $B$, we have that 
\[ -A_2 =  \frac{1}{2} \sum_{i,j=1}^n\int_{\R^n} \bar{d}_{ij}(\bar{u},_j v,_i - v,_j \bar{u},_i) dx. \] We have from Lemma~\ref{L1} that $\bar{u},_j v,_i - v,_j \bar{u},_i \in \mathcal{H}^1(\R^n)$ with \[ \|\bar{u},_j v,_i - v,_j \bar{u},_i\|_{\mathcal{H}^1(\R^n)} \leq C \|\nabla \bar{u}\|_{L_2(\R^n)} \|\nabla v\|_{L_2(\R^n)}, \] and since $BMO(\R^n)$ is the dual of the Hardy space $\mathcal{H}^1(\R^n)$, we derive that 
\[
|A_2| \leq C\|\bar{d}\|_{L_{\infty}(-1,0;BMO(\R^n;\R^{n\times n}))} \|\nabla \bar{u}\|_{L_2(\R^n)} \|\nabla v\|_{L_2(\R^n)},
\]
and a fortiori 
\[
|A_2| \leq C\|d\|_{L_{\infty}(-1,0;BMO(B;\R^{n\times n}))} \|\nabla u\|_{L_2(B)} \|\nabla v\|_{L_2(B)},
\]
(with $C$ depending only on $n$). Hence, we have that $\dvg g(\cdot,t) \in H^{-1}(B)$, with 
\[ \|\dvg g(\cdot,t)\|_{H^{-1}(B)} \leq C(n,a,d)\|\nabla u(\cdot,t)\|_{L_2(B)} \quad \mbox{for a.e $t\in (-1,0)$}.\] Therefore, there exists a unique $U(\cdot,t)\in \mathring{L}^1_2(B)$ which solves \eqref{E2} and such that \[ \|\nabla U(\cdot,t)\|_{L_2(B)} \leq C(n,a,d)\|\nabla u(\cdot,t)\|_{L_2(B)} \quad \mbox{for a.e $t\in (-1,0)$}.\] We also deduce that $\nabla U \in L_2(Q)$ and 
\[ \|\nabla U\|_{L_2(Q)} \leq C(n,a,d)\|\nabla u\|_{L_2(Q)}. \]
%%%%%%%%%%%%%%%%%%%%%%%%%%%%%%%%%%%%%%%                                              %%%%%%%%%%%%%%%%%%%%%%%%%%%%%%%%%%%%%%%%%%
\textbf{Step 2.} Now, we can rewrite \eqref{E3} as follows
\begin{equation}\label{E4}
\int_Q u\partial_t \phi dz = \int_Q \nabla U. \nabla \phi dz \quad \forall \phi \in C^{\infty}_0(Q).
\end{equation}
By a density arguments, we can test \eqref{E4} with functions $\phi(x,t) = \chi(t)\phi_k(x)$, where $\chi \in C^{\infty}_0(-1,0)$ and $\phi_k$ is an eigenfunction (introduced in the second part of the preliminaries section; here we choose $\Omega=B$). Since $(\phi_k)_{k=1}^{\infty}$ is a Hilbert basis of $L_2(B)$, we can write $u$ as follows
\[ u(x,t) = \sum_{k=1}^{\infty} d_k(t)\phi_k(x),\]where $d_k(t) = \int_B u(\cdot,t)\phi_k dx$; we also have
\[ U(x,t) = \sum_{k=1}^{\infty} b_k(t)\phi_k(x),\]where $b_k(t) = \int_B U(\cdot,t)\phi_k dx$. So we have, thanks to Lemma~\ref{L2}, that
\[ \|\nabla U\|^2_{L_2(Q)}=\int_{-1}^0 \sum_{k=1}^{\infty} b_k^2(t)\lambda_k dt \leq C(n,a,d)\|\nabla u\|^2_{L_2(Q)} < \infty .\]
We have now
\begin{align*}
\int_{-1}^0 d_k(t) \chi'(t) dt &= \int_{-1}^0\chi(t) \int_B \nabla U.\nabla \phi_k dx dt\\ 
&= \int_{-1}^0 \chi(t) b_k(t)\lambda_k dx dt.
\end{align*}
So, $d_k'(t) = -b_k(t) \lambda_k$ and we derive that
\[ \partial_t u(x,t) = \sum_{k=1}^{\infty} d_k'(t) \phi_k(x), \] where the convergence of this sum occurs in the space of distributions; thus we have, for every $w \in C^{\infty}_0(B)$ and $\chi\in C^{\infty}_0(-1,0)$, that 
\begin{align*}
\int_{-1}^0 \langle\partial u(\cdot,t), w\rangle\chi(t) dt &= -\lim_{N \to \infty} \int_{-1}^0\sum_{k=1}^N b_k(t)\lambda_k \int_B \phi_k(x) w(x) dx \chi(t) dt\\
&= \lim_{N \to \infty} \int_{-1}^0\sum_{k=1}^N b_k(t) \int_B \nabla\phi_k(x).\nabla w(x) dx \chi(t) dt\\
&\leq \|\nabla w\|_{L_2(B)} \int_{-1}^0\left(\sum_{k=1}^{\infty} b_k^2(t) \lambda_k\right)^{1/2}|\chi(t)| dt \quad \mbox{by Lemma~\ref{L2}}\\
&\leq c(n,a,d)\|\nabla u\|_{L_2(Q)}\|\chi\|_{L_2(-1,0)}\|\nabla w\|_{L_2(B)},
\end{align*}
and the statement follows.
\end{proof}

\begin{remark}\label{R1}
Let $\varphi \in C^{\infty}_0(B)$. Then, we readily deduce from the above proposition that
\[
\partial_t(u \varphi) \in L_2(-1,0;H^{-1}(B)).
\]
Since obviously $u \varphi \in L_2(-1,0;\mathring{L}^1_2(B))$, we conclude that \[ u\varphi \in C([-1,0];L_2(B)). \]
\end{remark}
%%%%%%%%%%%%%%%%%%%%%%%%%%%%%%%%%%%%%%%%%%%%%%%%%%%%%%%%%%%%%%%%%%%%%%%%%%%%%%%%%%%%%%%%%%%%%%%%%%%%%%%%%%%%%%%%%%%%%%%%%%
\begin{proof}[Proof of Theorem~\ref{T1}]
\textbf{Step 1.} Consider the following auxiliary equation
\begin{equation}\label{E5}
\begin{cases}
\partial_t w - \dvg(A\nabla w) = F - \dvg G\quad \mbox{in } B\times (-1,0),\\ 
\qquad \qquad \quad w|_{\partial' Q} = 0, 
\end{cases}
\end{equation}
where $F=u \partial_t \varphi - (A \nabla u).\nabla \varphi$ and $G = u(A\nabla \varphi)$. We have that the distribution $F-\dvg G$ belongs to $L_2(-1,0;H^{-1}(B))$. This is a direct consequence of the fact that $u\varphi$ is distributional solution of \eqref{E5} together with Remark~\ref{R1} and Step 1 of the proof of Proposition~\ref{P1}. But, for our purpose we are interested, more precisely, in the bounds of the terms which appear in the definition of $F-\dvg G$ and which belong to $L_2(-1,0;H^{-1}(B))$. Therefore let us consider a function $w \in C^{\infty}_0(B)$; then for a.e $t\in (-1,,0)$
\begin{align*}
\int_B F w dx + \int_B G.\nabla w dx &= \int_B u \partial_t \varphi w dx - \int_B (a\nabla u ).\nabla \varphi w dx - \int_B (d\nabla u ).\nabla \varphi w dx\\ &\mathrel{\phantom{=}}+ \int_B (a\nabla \varphi). \nabla w u dx + \int_B (d\nabla \varphi). \nabla w u dx\\
&= J_1 + J_2 + J_3,
\end{align*}
where
\[ J_1 := \int_B u \partial_t \varphi w, \]
\[ J_2 := - \int_B (a\nabla u ).\nabla \varphi w dx + \int_B (a\nabla \varphi). \nabla w u dx,\]
\[ J_3 := - \int_B (d\nabla u ).\nabla \varphi w dx + \int_B (d\nabla \varphi). \nabla w u dx. \]
We rewrite the term $J_3$ as follows
\begin{align*}
J_3 &= - \int_B d\nabla (u w) .\nabla \varphi dx + \int_B (d\nabla w). \nabla \varphi u dx+ \int_B (d\nabla \varphi). \nabla w u dx\\
&= - \int_B d\nabla (u w) .\nabla \varphi dx \quad \mbox{(by the skew symmetry of $d$)}.
\end{align*}
But again, thanks to the skew symmetry of $d$, we have that 
\[
J_3 = -\frac{1}{2}\sum_{i,j=1}^n \int_B b_{i,j}\left[ (u w),_j \phi,_i - \phi,_j (u w),_i \right] dx
\]
and making the same computations as for $A_2$ in Step 1 of the proof of Proposition~\ref{P1}, with the only difference being that we keep $p$ arbitrary (instead of choosing $p=2$ as in the proof of Proposition~\ref{P1}), we obtain
\[
|J_3| \leq C(n,a,d)\|\nabla(uw)\|_{L_p(B)} \|\varphi\|_{L_{p/(p-1)}(B)},
\]
(for $1<p<\infty$ to be suitably chosen in function of $n$) which implies that
\[
|J_3| \leq C(n,a,d,\varphi)\left( \|w\nabla u\|_{L_p(B)}+\|u\nabla w\|_{L_p(B)} \right).
\]
If $\mathbf{n\geq 3}$, we steadily have for 
\[ 1< p< \min(2,\frac{n}{n-1}), \]
that (here $2^* = \frac{2n}{n-2}$)
\begin{align*}
\|w\nabla u\|_{L_p(B)}+\|u\nabla w\|_{L_p(B)} &\leq c(n)\left[ \|w\|_{L_{2^*}(B)}\|\nabla u\|_{L_2(B)} + \|u\|_{L_{2^*}(B)}\|\nabla w\|_{L_2(B)}\right]\\
&\leq c(n) \left( \|\nabla u\|_{L_2(B)} + \|u\|_{L_{2}(B)}\right) \|\nabla w\|_{L_2(B)},
\end{align*}
where Sobolev embedding and Poincar\'{e}'s inequality are used in the last estimate.\\
\\
The case $\mathbf{n=2}$ is a straightforward adaptation of the previous (since $H^1(B)$ embeds continuously in every $L_s(B)$, $1\leq s<\infty$), whereas for the case $\mathbf{n=1}$, we take $p=2$, use the fact that $H^1(B)$ is continuously embedded in $L_{\infty}(B)$ and Poincar\'{e}'s inequality for the term in $w$. \\ \\
Next, we have the following easy bound for the terms $J_1$ and $J_2$:
\[ |J_1 + J_2| \leq C(n,\varphi) \left( \|\nabla u\|_{L_2(B)} + \|u\|_{L_{2}(B)}\right) \|\nabla w\|_{L_2(B)}. \]
So in conclusion, we have, for a.e $t\in (-1,0)$ that
\[
\|F(\cdot,t) - \dvg G(\cdot,t)\|_{H^{-1}(B)} \leq C(n,a,d,\varphi) \left( \|\nabla u(\cdot,t)\|_{L_2(B)} + \|u(\cdot,t)\|_{L_{2}(B)}\right),
\]
and we get a fortiori
\begin{equation}\label{E6}
\|F - \dvg G\|_{L_2(-1,0;H^{-1}(B))} \leq C(n,a,d,\varphi) \left( \|\nabla u\|_{L_2(Q)} + \|u\|_{L_{2,\infty}(Q)}\right)
\end{equation}
\\ \\ \textbf{Step 2.} Let us now tackle the question of well-posedness of \eqref{E5}. Consider the time-indexed family of bilinear forms
\[ \delta_t(w,v) := \int_B (A\nabla w).\nabla v dx. \]
Let us first notice that the map $t\in (-1,0)\mapsto \delta_t(w,v)$ is measurable for every $w,v \in \mathring{L}^1_2(B)$. Furthermore, we have by similar computations as those made in Step 1 in the proof of Proposition~\ref{P1}, that there exists a constant $C= C(n,a,d) >0$ independent of $t$ such that  
\[ |\delta_t (w,v)| \leq C\|\nabla w\|_{L_2(B)}\|\nabla v\|_{L_2(B)}, \]
for all $w,v \in \mathring{L}^1_2(B)$ i.e $\delta_t$ is a bounded bilinear operator on $\mathring{L}^1_2(B)$. We have, additionally, the following coercivity estimate
\begin{align*}
\phantom{{}\geq{}}
\delta_t (w,w) &= \int_B (A\nabla w).\nabla w dx\\
&= \int_B (a\nabla w).\nabla w dx\quad \mbox{(by the skew symmetry of $d$),}\\
&\geq \nu \int_B |\nabla w|^2 dx.
\end{align*}
In view of these previous estimates and the regularity proved for the right-hand side of \eqref{E5} and considering the evolution triple $\mathring{L}^1_2(B) \subset L_2(B) \subset H^{-1}$, we have by applying J-L. Lions abstract theorem for well-posedness of evolution equations (see e.g.,\cite{J-L65}, Theorem 4.1, Chapter 3, section 4) that there exists a unique solution 
\[ w\in  C([-1,0];L_2(B))\cap L_2(-1,0;\mathring{L}^1_2(B)), \] with \[ \partial_t w \in L_2(-1,0;H^{-1})\]
such that
\begin{equation}\label{E7}
\int_Q \partial_t w v dz + \int_Q (A\nabla w).\nabla v dz = \int_Q (F-\dvg G)v dz
\end{equation}
for any $v\in C^{\infty}_0(Q)$. Let us notice that, from Remark~\ref{R1}, the fact that $u\varphi$ is a distributional solution of \eqref{E5} and by the above uniqueness result:
\[ w = u\varphi. \]
On another hand, by the regularity obtained for $w$, we can extend identity \eqref{E7} to functions $v$ in $L_2(-1,0;\mathring{L}^1_2(B))$ and therefore, test \eqref{E7} with $w$ itself. Thus, we get
\begin{equation}\label{E8}
\int_Q (A \nabla w). \nabla w dz = \int_Q (F-\dvg G)w dz.
\end{equation}
Denote by $L$, the left-hand side of the above identity. By the skew symmetry of $d$, we obtain that
\[ L = \int_Q (a \nabla w). \nabla w dz; \] therefore coming back to $u$, we get
\begin{align*}
L &= \int_Q a  (\varphi \nabla u + u\nabla \varphi). (\varphi \nabla u + u\nabla \varphi) dz \\
&= \int_Q \varphi^2 (a\nabla u).\nabla u dz + 2\int_Q u \varphi (a \nabla u).\nabla \varphi dz + \int_Q u^2 (a\nabla \varphi).\nabla \varphi.
\end{align*}
Now, denote by $R$ the right hand side of \eqref{E8}; we easily obtain that 
\begin{align*}
R &= \int_Q u\partial_t \varphi w dz - \int_Q (a\nabla u).\nabla \varphi w dz + \int_Q (a\nabla \varphi).\nabla w u - \int_Q d \nabla (u w). \nabla \varphi dz\\
&=  \int_Q u^2 \varphi \partial_t \varphi dz - \int_Q u\varphi(a\nabla u).\nabla \varphi dz + \int_Q \varphi u (a\nabla u).\nabla \varphi dz + \int_Q u^2 (a\nabla \varphi).\nabla \varphi dz\\&\mathrel{\phantom{=}} -2 \int_Q u \varphi (d \nabla u).\nabla \varphi dz.
\end{align*}
Therefore, \eqref{E8} implies that 
\begin{equation}\label{E9}
\int_Q \varphi^2 (a \nabla u).\nabla u dz + 2\int_Q u \varphi (a\nabla u).\nabla \varphi dz = \int_Q u^2 \varphi \partial_t \varphi dz - 2\int_Q u \varphi (d \nabla u).\nabla \varphi dz, 
\end{equation}
for all $\varphi \in C^{\infty}_0(Q)$. Let us notice that all the integrals in the above identity are finite, especially the last one of the right hand side of \eqref{E9}. To see this we rewrite 
\[ \int_Q u \varphi (d \nabla u).\nabla \varphi dz = \int_Q d\nabla(u^2/2).\nabla (\varphi^2/2) dz \]and use the same method as in the estimation of $J_3$ in the previous step. \\ \\
\textbf{Step 3.} Now, we choose $\varphi(x,t) = \chi_\epsilon(t) \phi(x,t)$, where $\phi \in C^{\infty}_0(B\times (-1,1))$ and $\chi_\epsilon(t) = 1$ if $t\leq t_0 -\varepsilon$, $\chi_\epsilon(t) = (t_0+\epsilon-t)/(2\epsilon)$, if $t_0-\epsilon < t< t_0+\epsilon$, and $\chi_\epsilon(t)=0$ when $t\geq t_0 +\epsilon$, with $t_0 \in (-1,0)$. Therefore passing to the limit $\epsilon \to 0$ in \eqref{E9},  we have that Theorem~\ref{T1} is proved.

\end{proof}
%%%%%%%%%%%%%%%%%%%%%%%%%%%%%%%%%%%%%%%%%%%%%%%%%%%%%%%%%%%%%%%%%%%%%%%%%%%%%%%%%%%%%%%%%%%%%%%%%%%%%%%%%%%%%%%%%%%%%%%%%%
\paragraph{Acknowledgement}
This work was supported by the Engineering and Physical Sciences Research Council [EP/L015811/1]. The author would like to thank Siran Li and Ghozlane Yahiaoui for their careful reading and insightful discussion. 
%%%%%%%%%%%%%%%%%%%%%%%%%%%%%%%%%%%%%%%%%%%%%%%%%%%%%%%%%%%%%%%%%%%%%%%%%%%%%%%%%%%%%%%%%%%%%%%%%%%%%%%%%%%%%%%%%%%%%%%%%%

\end{document}